\documentclass[12pt]{amsart}
\usepackage{amssymb}
\usepackage{graphicx}
\usepackage{fullpage}
\usepackage{amsmath}
\usepackage{amsthm}
\usepackage{float}
\usepackage{youngtab}
\numberwithin{equation}{section}
\newtheorem{theorem}{Theorem}
\numberwithin{theorem}{section}
\newtheorem{lemma}{Lemma}
\numberwithin{lemma}{section}
\newtheorem{cor}{Corollary}
\numberwithin{cor}{section}

\numberwithin{remark}{section}
\numberwithin{figure}{section}

\DeclareMathOperator{\maj}{maj}
\DeclareMathOperator{\inv}{inv}
\DeclareMathOperator{\Swap}{Swap}

\title{A generalized Major index statistic on tableaux}

\author{James Haglund and Emily Sergel}

\date{}

\begin{document}

\maketitle

\begin{abstract}

We extend the family of statistics $\maj_d$, introduced for permutations by Kadell \cite{kadell}, to standard Young tableaux. At one extreme, we have the traditional Major index statistic $\maj_1$ for tableaux. At the other end, whenever $N \geq n-1$, then $\maj_{N} = \inv$, the inversion statistic introduced by \cite{HSinv}. This is answers a question of Assaf \cite{assaf}, who defined $\maj_2$ and $\maj_3$ for tableaux.

\end{abstract}

%%%%%%%%%%%%%%%%%%%%%%%%%%%%%%%%

\section{Permutation statistics} \label{sec:perm}

Let $\sigma = \sigma_1 \sigma_2 \dots \sigma_n$ denote the permutation of $\{1,2,\dots,n\}$ sending $i$ to $\sigma_i$. The set of inversions of a permutation $\sigma$ is
$$
Inv(\sigma) = \left\{ (i,j) : i<j \hbox{ and } \sigma_{i} > \sigma_{j} \right\}.
$$
Let $\inv(\sigma) = | Inv(\sigma) |$. In other words,
$$
\inv(\sigma) = \sum_{i<j} 1 \cdot \chi\left( (i,j) \in Inv(\sigma) \right)
$$
where $\chi$ of a statement $A$ is 1 when $A$ is true and 0 when $A$ is false.
MacMahon \cite{major} introduced another statistic $\maj$ which has the same distribution as $\inv$ on permutations. That is,
$$
\sum_{\sigma \in S_n} q^{\inv(\sigma)} = \sum_{\sigma \in S_n} q^{\maj(\sigma)}.
$$
This statistic is also based on inversions, but assigns different weights to different inversions. In particular,
\begin{align*}
\maj(\sigma) &= \sum_{i=1}^{n-1} i \cdot \chi\left( (i,i+1) \in Inv(\sigma)\right)\\
&= \sum_{i<j} i \cdot \chi\left( j=i+1\right) \cdot \chi\left( (i,j) \in Inv(\sigma)\right)
\end{align*}
Many years after MacMahon introduced the $\maj$ statistic, Foata \cite{foata} gave an explicit map $f$ with $\maj(f(\sigma)) = \inv(\sigma)$ for all permutations $\sigma$.

Kadell \cite{kadell} extends these two statistics naturally using an upper triangular matrix of weights $W = || w_{i,j} ||_{1 \leq i,j \leq n}$. For each such $n \times n$ matrix and $\sigma \in S_n$, he defines
$$
\inv_W(\sigma) = \sum_{i<j} w_{i,j} \cdot \chi\left( (i,j) \in Inv(\sigma) \right).
$$
The statistics of interest here correspond to the matrices $W^{(d)} = || w^{(d)}_{i,j} ||$ with $d>0$ where
$$
w^{(d)}_{i,j} = \begin{cases}
0 & \hbox{if }i \geq j \hbox{ or }j-i>d,\\
1 & \hbox{if }j >i \hbox{ and }j-i<d,\\
i & \hbox{if }j-i=d.\\
\end{cases}
$$
This statistic $\inv_{W^{(d)}}$ was called $\inv_d$ by Kadell, and later reintroduced by Assaf \cite{assaf} as $\maj_d$. Both Kadell and Assaf study these statistics because they are closely related to LLT polynomials and Macdonald polynomials (see \cite{assaf} for more details). Here, we follow Assaf's notation. That is, we set
$$
\maj_d(\sigma) = \inv_{W^{(d)}}(\sigma) = \sum_{i<j<i+d} 1 \cdot \chi\left( (i,j) \in Inv(\sigma)\right) \, + \sum_{j=i+d} i \cdot \chi\left( (i,j) \in Inv(\sigma)\right).
$$

\begin{theorem}[Kadell \cite{kadell}] \label{thm:permequidist}
For any $n,d>0$,
$$
\sum_{\sigma \in S_n} q^{\inv(\sigma)} = \sum_{\sigma \in S_n} q^{\maj_d(\sigma)}
$$
\end{theorem}

Kadell's proof of Theorem~\ref{thm:permequidist} relies on a broad family of transformations on the weight matrix $W$ which preserve distributions. Then these transformations are combined to obtain maps exchanging $\inv$ for $\maj_d$ which interpolate between the identity map ($d=1$) and Foata's map ($d=n-1$). For a particular $d$, Kadell's map iteratively breaks the permutation into blocks and then cycles numbers the top $d$ numbers. In the next section, we generalize this map to tableaux and, from the map, derive our $\maj_d$ statistic.

%%%%%%%%%%%%%%%%%%%%%%%%%%%%%%%%

\section{Tableau Statistics} \label{sec:tab}

Let $\lambda = (\lambda_1 \geq \lambda_2 \geq \dots \geq \lambda_k > 0)$ with $\lambda_1+\lambda_2+ \dots +\lambda_k=n$. We say that $\lambda$ is a partition of $n$, written $\lambda \vdash n$. Partitions appear in the theory of symmetric functions, e.g., as indices for the space of homogeneous symmetric functions. For an introduction to symmetric function theory, see Stanley \cite{EC2}. Associated to each partition $\lambda$ is a diagram called the Young diagram or Ferrers diagram of $\lambda$. Following the French convention, we define the Young diagram of $\lambda$ to be a collection of boxes which is left aligned and whose $i$-th row from the bottom contains $\lambda_i$ boxes. For example, see Figure~\ref{fig:Young}. We will often abuse notation by identifying a partition with its diagram.

\begin{figure}[H]
\begin{center}
\includegraphics[width=0.9in]{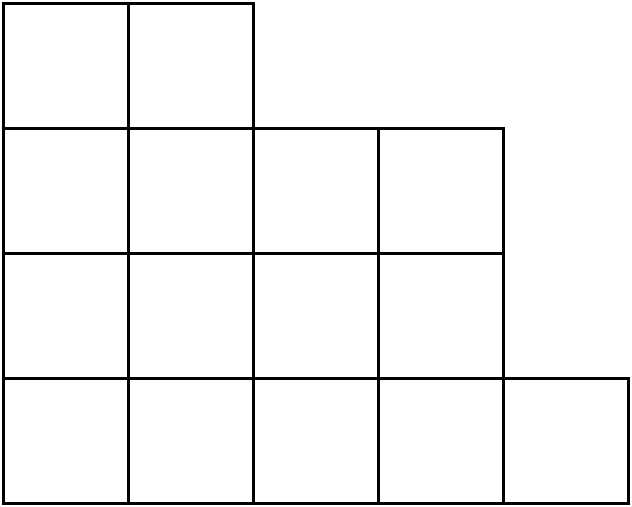}
\end{center}
\vspace{-10pt}
\caption{The Young diagram of the partition $(5,4,4,2)$.}
\label{fig:Young}
\end{figure}

A filling of such a diagram is a map between the cells of the diagram and some set of labels, usually positive integers. If $\lambda \vdash n$, a standard Young tableaux of shape $\lambda$ is a filling of the diagram of $\lambda$ with the labels $\{1,2,\dots,n\}$ so that each label is used exactly once and labels increase as you move up a column or to the right in a row. For example, there are five standard Young tableaux of shape $(3,2)$. They are pictured in Figure~\ref{fig:sty32}. The collection of standard Young tableaux of any shape $\lambda$ is denoted $STY(\lambda)$.

\begin{figure}[H]
\begin{center}
\includegraphics[width=3.5in]{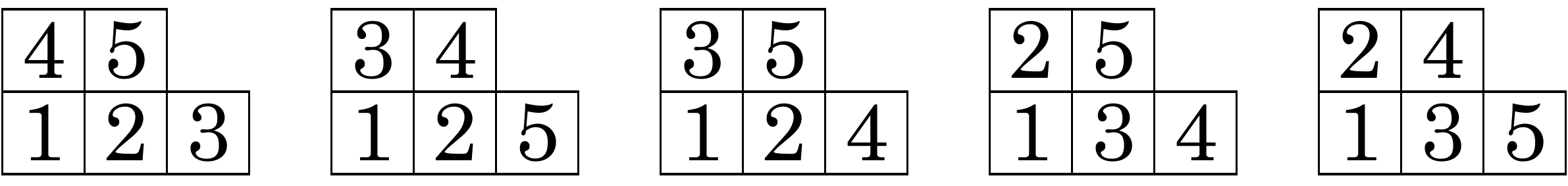}
\end{center}
\vspace{-10pt}
\caption{All standard Young tableaux of shape $(3,2)$.}
\label{fig:sty32}
\end{figure}

There is a very natural analog of the Major index statistic for tableaux. For any standard young tableaux $T$, let
$$
\maj(T) = \sum_{i=1}^{n-1} i \cdot \chi\left( \hbox{$i+1$ appears in a higher row of $T$ than $i$} \right).
$$
This statistic is compatible with the $\maj$ statistic on permutations and the Robinson-Schensted-Knuth algorithm (see Stanley \cite{EC2} for the definition and relevance of this algorithm). Haglund and Stevens \cite{HSinv} define an inversion statistic on tableaux which is equidistributed with $\maj$. Our goal here is to define a $\maj_d$-type statistic which is also equidistributed with $\maj$ and $\inv$ on standard Young tableaux. Standard Young tableaux and their statistics are also very important in symmetric function theory, particularly in the theory of Macdonald polynomials and LLT polynomials. See Macdonald \cite{macbook} and Assaf \cite{assaf} for more details.

Haglund and Stevens' point of departure is to suppose that a pair of labels $i<j$ in a standard Young tableaux $T$ should always make an inversion when $i$ is in a strictly lower row and weakly further right than $j$, but never when $i$ is in a weakly higher row. 

\begin{figure}[H]
\begin{center}
\includegraphics[width=0.9in]{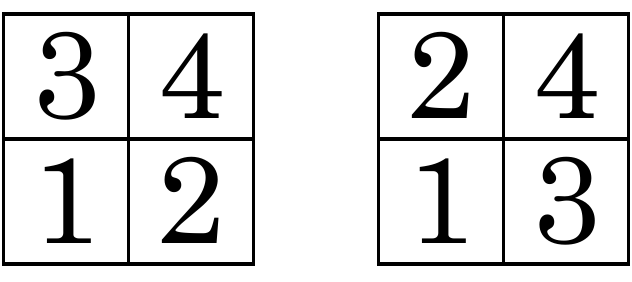}
\end{center}
\vspace{-10pt}
\caption{All standard Young tableaux of shape $(2,2)$.}
\label{fig:sty22}
\end{figure}

For example, consider the two standard Young tableaux of shape $(2,2)$ in Figure~\ref{fig:sty22}. The tableau $T^{(2,2)}_1$ on the left has $\maj=2$ and the tableau $T^{(2,2)}_2$ on the right has $\maj=4$. So we would like to define the set of inversions $Inv(T)$ of a tableau $T$ so these two tableaux have $Inv$ of sizes 2 and 4, in some order. Based on the heuristic above, we would like to define $Inv$ so that $(1,3), (2,3), (2,4) \in Inv(T^{(2,2)}_1)$ and $(1,2), (3,4) \in Inv(T^{(2,2)}_2)$. So this can only be achieved if $(1,4) \in Inv(T^{(2,2)}_1)$ and $(1,4) \not\in Inv(T^{(2,2)}_2)$.

To resolve this problem, Haglund and Stevens construct a path starting at each cell of a tableau $T$ and define $Inv(T)$ to be the set of pairs $(i,j)$ with $i<j$ and $i$ lying below the path which starts at $j$'s cell in $T$. For $\maj_d$, we will also define such paths. But for us, the path will depend on $d$, as well as on the tableaux. That is, a particular tableau does not have a fixed set of inversions which are being weighted differently by different statistics. For Haglund and Stevens, these paths determine the inversion pairs which all have weight 1. Ideally, the weight of a particular inversion $(i,j)$ in a $\maj_d$-type statistic would depend only on the values of $i$, $j$, and $d$. But for our $\maj_d$ statistic, the weight of a particular inversion will be more involved. These additional complexities seem to be necessary.

%%%%%%%%%%%%%%%%

\subsection{The inversion statistic of \cite{HSinv}} \label{sec:tabinv}

Let $\lambda$ be a partition of $n$ and let $T$ be a standard Young tableaux of shape $\lambda$. Let $1\leq k \leq n$ and define the path $\pi(T,k)$ as follows: Start at the lower left-hand corner of the cell of $T$ containing $k$. At each step, compare the entries directly left of and below the current location. If both exist, move one step in the direction of the larger entry. If one of these entries doesn't exist, move toward the one that does. If both don't exist, stop. For example, consider the first four tableaux in Figure~\ref{fig:HSinvex}. For each of these tableaux $T$, the path $\pi(T,k)$ is drawn in a thick line for $k=9,8,7,5$, respectively.

\begin{figure}[H]
\begin{center}
\includegraphics[width=3.5in]{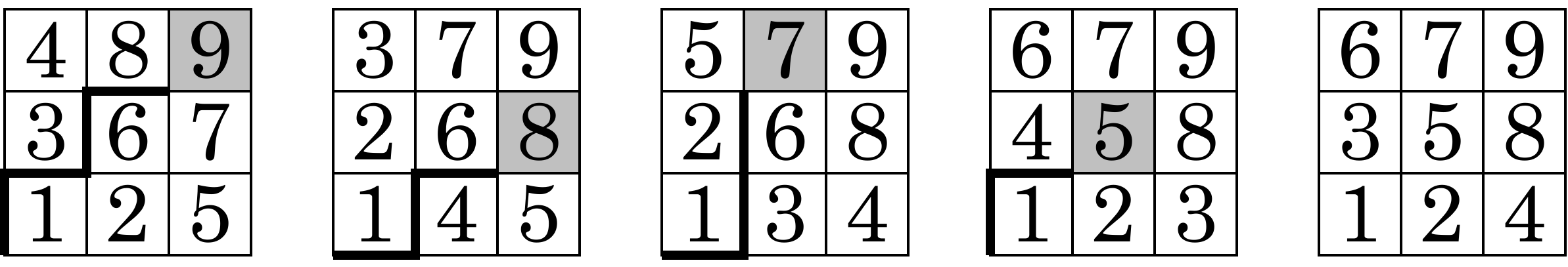}
\end{center}
\vspace{-10pt}
\caption{The computation of $\inv$ for a standard Young tableaux of shape $(3,3,3)$.}
\label{fig:HSinvex}
\end{figure}

Given a tableau $T$ and a path $\pi$ starting at the cell labeled $k$, we break the numbers of $1,2,\dots,k-1$ into disjoint blocks as follows. Each block consists of consecutive numbers $a,a+1,\dots,a+m$ so that $a$ is on the same side of $\pi$ as 1 and all of the numbers $a+1,\dots,a+m$ are on the opposite side. These blocks are maximal in the sense that either $a+m+1$ is on the same side of $\pi$ as 1, or $a+m+1=k$. In the first tableau of Figure~\ref{fig:HSinvex}, the blocks for the given path are $\{1\}, \{2,3,4\}, \{5\}, \{6\}, \{7,8\}$. For the second tableau, they are $\{1\}, \{2\}, \{3,4,5\}, \{6\}, \{7\}$.

Notice that for any a tableau $T$, in the blocks about the path $\pi(T,k)$, the smallest element will never be adjacent to any of the bigger elements from the same block. This is because of the nature of the path separating the first element from the others. For a detailed explanation, see \cite{HSinv}. Hence in each block we may rotate the elements - that is we may write $a$ into $a+1$'s cell, $a+1$ into $a+2$'s cell, \dots, and $a+m$ into $a$'s cell - without violating the row- and column-increasing conditions for Young tableaux. Let $\Psi_k(T)$ be the standard Young tableau obtained in this way. For example, if the first tableau of Figure~\ref{fig:HSinvex} is called $S$, then the others are $\Psi_9(S)$, $\Psi_8 \circ \Psi_9(S)$, $\Psi_7 \circ \Psi_8 \circ \Psi_9(S)$, and $\Psi_5 \circ \Psi_7 \circ \Psi_8 \circ \Psi_9(S)$.

Haglund and Stevens show that the maps $\Psi_k$ are bijective and give a detailed description of their inverse. They also show that $\maj(\Psi_1 \circ \Psi_2 \circ \cdots \Psi_n(T))$ equals $\inv(T) = |Inv(T)|$ when $Inv(T)$ is defined as follows. For each cell, there will be exactly one $j$ for which $\Psi_{j+1} \circ \cdots \circ \Psi_n(T)$ contains $j$. Associate the path $\pi(\Psi_{j+1} \circ \cdots \circ \Psi_n(T),j)$ to this cell. Then a pair $(i,j)$ with $i<j$ forms an inversion of $T$ if and only if $i$ is below the path associated to the cell containing $j$ in $T$.

We consider again the first tableau $S$ of Figure~\ref{fig:HSinvex} and give a detailed account of the computation of $\inv(S)$. The number $9$ appears in the top right cell, which attacks five other cells. Hence we have $(1,9), (2,9), (5,9), (6,9), (7,9) \in Inv(S)$. Then the number $8$ appears in the middle right cell of $\Psi_9(S)$ and attacks two other cells. Back in $S$, the middle right cell contains $7$ and the cells it attacks contain $2$ and $5$, so we have $(2,7), (5,7) \in Inv(S)$. Next $7$ appears in the top middle cell of $\Psi_8 \circ \Psi_9(S)$ and attacks 4 cells. Looking back in $S$, this gives $(2,8), (5,8), (6,8), (7,8) \in Inv(S)$.

Now the $6$ appears in the top left cell of $\Psi_7 \circ \Psi_8 \circ \Psi_9(S)$. Since it is already on the left boundary, the path starting there goes straight down. In $S$, the top left cell contains a $4$ and some of the cells falling under the path are larger than $4$ and hence do not contribute inversions. From this cell we only gain three inversions - $(1,4), (2,4), (3,4) \in Inv(T)$. Back in $\Psi_7 \circ \Psi_8 \circ \Psi_9(S)$, every number less than $6$ is below the path starting at $6$. This means every block will have a single element and $\Psi_6 \circ \Psi_7 \circ \Psi_8 \circ \Psi_9(S) = \Psi_7 \circ \Psi_8 \circ \Psi_9(S)$. This happens often - whenever $k$ is in a cell on the left or bottom border of $T$, $\Psi_k(T)=T$.

Continuing in this fashion, we find that the remaining inversions of $S$ are $(1,6)$, $(2,6)$, $(5,6)$, $(1,3)$, and $(2,3)$. Hence $\inv(S) = |Inv(S)| = 19$. And indeed $\Psi_1 \circ \Psi_2 \circ \cdots \Psi_9(S)$, which is the rightmost tableau in Figure~\ref{fig:HSinvex}, has $\maj = 19$.

%%%%%%%%%%%%%%%%

\subsection{The generalized Major index} \label{sec:tabmaj}

Let $T \in SYT(\lambda)$ for some $\lambda \vdash n$ and let $a+1 < k \leq n$. Then define $\Swap_{a,k}(T)$ to be the tableau obtained by switching $a$ and $a+1$ if they're separated by $\pi(T,k)$ and otherwise doing nothing ($\Swap_{a,k}(T)=T$). Note that $\Psi_k(T) = \Swap_{1,k} \circ \dots \circ \Swap_{k-2,k}(T)$.

\begin{lemma} \label{lemma:swap}
For $T \in SYT(\lambda)$ and $a+1 < k \leq n$, $\Swap_{a,k}(T) \in SYT(\lambda)$.
\end{lemma}

\begin{proof}
Simply note that if the labels $a$ and $a+1$ are adjacent in $T$, then the path $\pi(T,k)$ will not separate them. In particular, if $a$ is directly left of $a+1$, then either there is no cell above $a$, or it contains a label larger than both $a$ and $a+1$. Hence if the path $\pi(T,k)$ ever reaches the upper right corner of $a$'s cell, it will more left, leaving both $a$ and $a+1$ below the path. Similarly, if $a+1$ is above $a$, then the cell to $a$'s right contains a larger label (if it exists) so the path $\pi(T,k)$ will not separate them. If $a$ and $a+1$ are not adjacent, then their labels can be switched without violating the row or column increasing conditions for standard Young tableaux.
\end{proof}

Suppose $k \geq d$. Let $\Psi_k^{(d)}(T) = \Swap_{\max(k-d,1),k} \circ \dots \circ \Swap_{k-2,k}$. In other words, $\Psi_k^{(d)}(T)$ is the tableau obtained from standard Young tableau $T$ by cycling the numbers $\max(k-d,1),\dots,k-1$ around the path $\pi(T,k)$ in maximal blocks $\{a,a+1,\dots,a+m\}$ so that $a$ is on the same side of $\pi(T,k)$ as $\max(k-d,1)$ and $a+1,\dots,a+m$ are on the other side.

For example, each of the first four tableaux in Figure~\ref{fig:cycleExs} is followed by its image under $\Psi_k^{(4)}$ where $k$ is the label in the gray cell. The operator $\Psi_k(T)$ from Section~\ref{sec:tabinv} is equal to $\Psi_k^{(n-1)}(T)$. Since $\Psi_k^{(d)}$ is a composition of $\Swap$ operators, the image of a standard Young tableau is also a standard Young tableau by the Lemma.

\begin{figure}[H]
\begin{center}
\includegraphics[width=3.5in]{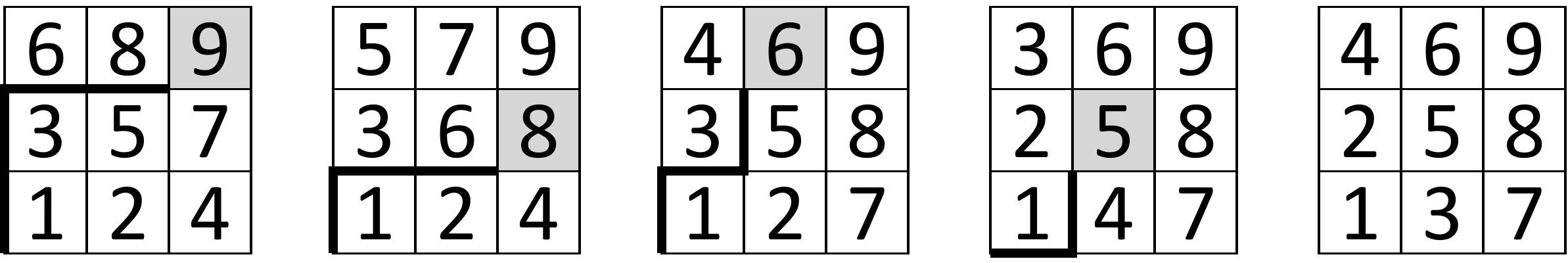}
\caption{The computation of $\maj_4$ for a tableau of shape $(3,3,3)$.}
\label{fig:cycleExs}
\end{center}
\end{figure}

In the remainder of this section, we develop our statistic $\maj_d$ by using the map $\Psi^{(d)} = \Psi_1^{(d)} \circ \dots \circ \Psi_n^{(d)}$ as follows. First, we will see that each map $\Psi_k^{(d)}$ is a bijection on standard Young tableaux of a fixed shape. This guarantees that $\maj_1(\Phi^{(d)}(T))$ is equidistributed with $\maj(T)$ as $T$ ranges over $SYT(\lambda)$. Careful analysis of the maps $\Psi^{(d)}_k$ will then yield an description of $\maj_d(T)$ given by weighted inversions of $T$ which will equal $\maj_1(\Phi^{(d)}(T))$.

We show that each map $\Phi_k^{(d)}$ is bijective by constructing its inverse. We will make use of the inverse $\Phi_k$ of $\Psi_k$ constructed by Haglund and Stevens. Its construction can be found immediately before Theorem 4.5 of \cite{HSinv}. For any standard Young tableau $T$, let $\Phi_k^{(d)}(T)$ be the tableau obtained from $S=\Phi_k(T)$ by cycling the numbers $\{1,\dots,k-d\}$ about the path $\pi(S,k)$. In other words, apply the operators $\Swap_{a,k}$ for $a=1,2,\dots,k-d-1$ to $S$ successively. Clearly the result is also in $SYT(\lambda)$.

\begin{theorem}
Let $S,T \in SYT(\lambda)$ such that $T = \Psi_k^{(d)}(S)$. Then $S = \Phi_k^{(d)}(T)$.
\end{theorem}

\begin{proof}
If $k-d < 0$, then the theorem is trivial. Suppose $k>d$. By definition, $\Phi_k^{(d)}(T) = \Psi_k^{(d)}( \Phi_k^{(d)}(S) ) = \left( \Swap_{k-2,k} \circ \cdots \circ \Swap_{\max(k-d,1),k} \right) \circ \left( \Swap_{\max(k-d,1),k} \circ \dots \circ \Swap_{1,k} \right) \circ \Phi_k(S) = \Psi_k \circ \Phi_k(S) = S$.
\end{proof}

Now our goal is to define an appropriate statistic $\maj_d$ so that $\maj_d(T) = \maj_1(\Psi^{(d)}(T))$. Let $T_0 = T$ and inductively define $T_{i+1} = \Psi^{(d)}_{n-i+1}(T_i)$. Then $T_n = \Psi^{(d)}(T)$. For example, if $d=4$ then the tableaux in Figure~\ref{fig:cycleExs} are, from left to right, $T_0$, $T_1$, $T_2{=}T_3$, $T_4$ and $T_5{=}\cdots{=}T_9$.

\begin{lemma}
Let $d \leq n$ and $T \in SYT(\lambda)$ for some $\lambda \vdash n$. The label $\max(n-d,1)$ lies under the path $\pi(T,n)$ iff $n-1$ and $n$ form a descent in $\Psi^{(d)}(T)$.
\end{lemma}

\begin{proof}
Note that in $\Psi^{(d)} = \Psi_1^{(d)} \circ \dots \circ \Psi_n^{(d)}$, the position of $n$ is fixed and only $\Phi^{(d)}_n$ affects the position of $n-1$.

First suppose that $\max(n-d,1)$ lies under $\pi(T,n)$. Then we apply $\Psi_n^{(d)} = \Swap_{n-2,n} \circ \dots \circ\Swap_{\max(n-d,1),n}$. At the first step, $\max(n-d,1)$ is under the path. Hence after applying $\Swap_{\max(n-d,1),n}$, we will have that $\max(n-d,1)+1$ is under the path. In fact, whenever it is time to apply the operator $\Swap_{a,n}$, we will have that $a$ is under the path. Hence when we apply the last $\Swap$ operator, either both $n-2$ and $n-1$ will be under the path (and stay that way) or $n-1$ will be above the path and switch places with $n-2$. Hence $n-1$ will be below the path $\pi(T,n)$ and therefore in a row which is lower than $n$'s row, forming a descent as desired.

Similarly, if $\max(n-d,1)$ is above the path, then at the step where we apply $\Swap_{a,n}$, we will have that $a$ is above the path. The end result will be that $n-1$ is above the path $\pi(T,n)$. Everything in a lower row than $n$ and weakly to the right must be under $\pi(T,n)$. But $n-1$ cannot be in a row below $n$ and strictly to its left without violating the row and column increasing properties of standard Young tableaux; if it were, there would have to be another label $x$ in the cell at the intersection of $n-1$'s row and $n$'s column for which $n-1 \leq x < n$. Hence $n-1$ cannot be in a lower row than $n$ and therefore they cannot form a descent.
\end{proof}

Therefore, since we would like $\maj_d(T) = \maj_1( \Psi^{(d)}(T) )$, our statistic will have to satisfy the following recursion.
\begin{equation} \label{eq:recur}
\maj_d(T) = \begin{cases} \maj_d(\Phi_n^{(d)}(T)-\{n\})+n-1 & \hbox{if }\max(n-d,1)\hbox{ is under }\pi(T,n)\\
\maj_d(\Phi_n^{(d)}(T)-\{n\}) & \hbox{else} \end{cases}
\end{equation}

It is tempting to simply define inversions following \cite{HSinv} and then apply the same weights to these inversions as we would for permutations: If $j-i>d$, inversion $(i,j)$ gets weight 0. If $j-i=d$, it gets weight $d$. Otherwise it gets weight 1. However, this does not satisfy the recursion above (and is not equidistributed with $\maj_1$).

For example, consider the tableau $T$ on the left side of Figure~\ref{fig:cycleExs}. Following this rule, we would get the following inversions with positive weights: $(5,9)$, $(7,9)$, $(4,8)$, $(5,8)$, $(7,8)$, $(2,5)$, $(1,6)$, $(2,6)$, $(3,6)$, $(4,6)$, $(5,6)$, $(1,3)$, and $(2,3)$ for a total weight of $21$. Compare this with $\Phi_9^{(4)}(T)-\{9\}$. (Note that the second tableau in the figure is $T_1=\Phi_9^{(4)}(T)$.) Here we would get a total weight of $15$ coming from the inversions $(4,8)$, $(4,7)$, $(6,7)$, $(2,6)$, $(4,6)$, $(1,5)$, $(2,5)$, $(3,5)$, $(4,5)$, $(1,3)$, $(2,3)$. This necessitates the more subtle definition found below.

We now define $\maj_d(T)$ as a sum of weighted inversions. Note that if $k \geq d$ then the label $k-d$ will lie in the same cells of $T$ and $T_{n-k}$. Furthermore $1$ is always in the same cell - the lower left corner. Let $x$ be the cell which contains the label $k$ in $T_{n-k}$. Let $\ell$ be the label of cell $x$ in $T$. Then if $x$ lies under the path $\pi(T_{n-k},k)$, the pair $(\max(k-d,1),\ell)$ gets weight $\max(k-d,1)$. Otherwise it gets weight 0. For any $n-k < s < \ell$, the pair $(s,\ell)$ gets weight 1 if $s$ lies under $\pi(T_{n-k},k)$ and 0 otherwise. All pairs $(s,\ell)$ with $s<n-k$ get weight 0 too.

For example, when $T$ is the leftmost tableau in Figure~\ref{fig:cycleExs} and $d=4$, the following inversions pairs will have positive weights: $(5,9)$, $(7,9)$, $(4,7)$, $(2,8)$, $(4,8)$, $(5,8)$, $(7,8)$, $(2,5)$, $(4,5)$, $(1,6)$, $(2,6)$, $(3,6)$, $(4,6)$, $(5,6)$, $(1,3)$, and $(2,3)$. All of these get weight 1 except for $(5,9)$, $(4,7)$, and $(2,8)$, which get weights $5$, $4$, and $2$, respectively. Hence $\maj_4(T) := 24$.

\begin{theorem} \label{thm:equid}
For any $d \leq n$ and $T \in SYT(\lambda)$ for some $\lambda \vdash n$, we have
$$
\maj_d(T) = \begin{cases} \maj_d(\Phi_n^{(d)}(T)-\{n\})+n-1 & \hbox{if }\max(n-d,1)\hbox{ is under }\pi(T,n)\\
\maj_d(\Phi_n^{(d)}(T)-\{n\}) & \hbox{else} \end{cases}
$$
Hence $\maj_d(T) = \maj_1(\Psi^{(d)}(T))$.
\end{theorem}

\begin{cor}
For any $n$ and any $\lambda \vdash n$, the statistics $\{\maj_d : 1 \leq d \leq n\}$ are equidistributed on $SYT(\lambda)$.
\end{cor}

\begin{proof}[Proof of Theorem~\ref{thm:equid}]
Suppose that $k<n$ and we have an inversion pair $(s,b)$ in $T$ so that $s=\max(d-k,1)$ and it $t$ lies in the same cell of $T$ as $k$ does in $T_{n-k}$. Since $\Phi_n^{(d)}$ does not affect the label $s$, we will assign the same weight, $s$, to this inversion for both $\maj_d(T)$ and $\maj_d(\Phi_n^{(d)}(T)-\{n\})$. 

First consider the case in which $m:=\max(n-d,1)$ is under $\pi(T,n)$. Then all of the inversions $(s,n)$ for which $m\leq s<n$ and $s$ is under $\pi(T,n)$ will be lost when we induct. If we can show that we lose an additional inversion for each $m<s<n$ when $s$ is above the path - and nothing else - then we will have the desired equality. But indeed the sets of cells attacking one another do not change here: they are defined in terms of paths in later tableaux $T_1, T_2, \dots$. And the pairs with weight greater than one also do not change, as we have already observed. So we only need to see when a pair with weight 1 becomes a pair with weight 0 or vice versa. When we apply $\Phi_n^{(d)}$, such changes can only occur within a single block $\{a,a+1,\dots,a+m\}$ where $a$ is separated from the $a+i$'s by $\pi(T,n)$. The map will send $a$ to $a+m$ and lower each of the other labels by 1. This means that any inversions between $a$ and the $a+i$'s will be broken and no new inversions will be created. Hence we will lose one additional weight for each cell above $\pi(T,n)$ whose label $s$ satisfies $m<s<n$.

On the other hand, if $m$ is above $\pi(T,n)$, then we need to see that the weight coming from inversions previously made with $n$ will be replaced by something else. Since $(m,n)$ has weight 0, this all comes from inversions $(s,n)$ with $m<s<n$ and weight $1$. That is, this is the number of cells below $\pi(T,n)$ whose label $s$ satisfies $m<s<n$. When we apply $\Phi_n^{(d)}$, we may change a weight of 0 to 1 or 1 to 0, as before. Again, such a change can only happen within a single block $\{a,a+1,\dots,a+m\}$. Now $a$ is below $a+1,\dots,a+m$ and cycling will replace $a$ by $a+m$ and reduce all the other labels. Hence the cell originally containing $a$ will make an inversion with each of the other cells. This means we get one additional inversion for each label $m<s<n$ where $s$ lies above $\pi(T,n)$ as desired.
\end{proof}

%%%%%%%%%%%%%%%%%%%%%%%%%%%%%%%%

\nocite{*}
\bibliographystyle{alpha}

\end{document}